\documentclass{birkjour}
\usepackage{amssymb,amsmath}
\usepackage[dvipsnames]{xcolor}
\usepackage[colorlinks,hyperindex]{hyperref}
\usepackage{graphicx}
\usepackage{tikz}
\usetikzlibrary{arrows,calc,shapes,decorations.pathreplacing}
\usepackage{bbm}
\usepackage{caption}
\usepackage{faktor}
\usepackage{pdfpages}
\usepackage{url}
\usepackage[all]{xy}
\usepackage{hyperref}
\usepackage{xcolor}
 \newtheorem{thm}{Theorem}[section]

 \newtheorem{prop}[thm]{Proposition}
 \theoremstyle{definition}
 \newtheorem{defn}[thm]{Definition}
 \theoremstyle{remark}
 \newtheorem{rem}[thm]{Remark}
 
 \numberwithin{equation}{section}

\begin{document}

%
%
%
%
%
%
%
%
%

\title{The Graf product: a Clifford structure  framework   on the exterior bundle}

\author{R. Lopes}

\address{Centro de Matem\'atica, Computa\c c\~ao e Cogni\c c\~ao\br Universidade Federal do ABC - UFABC\br  09210-580, Santo Andr\'e\br Brazil}

\email{rian.lopes@ufabc.edu.br}

\author{R. da Rocha}
\address{Centro de Matem\'atica, Computa\c c\~ao e Cogni\c c\~ao\br Universidade Federal do ABC - UFABC\br  09210-580, Santo Andr\'e\br Brazil}
\email{roldao.rocha@ufabc.edu.br}


\date{January 1, 2004}

\newcommand{\blkdiam}{\tikz\node[rectangle, draw, yscale=1, scale=.47, rotate=45, fill=black] {};}

\newcommand{\cliff}{{\mathcal{C}\ell}_{p, q}}

\newcommand{\M}{\mathcal{M}}

\newcommand{\tangb}{T^{*}\mathcal{M}}

\newcommand{\tanfib}{T^{*}_x\mathcal{M}}

\newcommand{\extbund}{{\bigwedge}(T^{*}\mathcal{M})}

\newcommand{\extevbund}{{\bigwedge}^{+}(T^{*}\mathcal{M})}		

\newcommand{\extfib}{{\bigwedge}(T^{*}_x \mathcal{M})}	

\newcommand{\extfibk}{{\bigwedge^k}(T^{*}_x \mathcal{M})}			

\newcommand{\extevfib}{{\bigwedge}^+(T^{*}_x \mathcal{M})}	

\newcommand{\clifbund}{{\mathcal{C}\ell}(T^{*}\mathcal{M})}	

\newcommand{\clifevbund}{{\mathcal{C}\ell}^+(T^{*}\mathcal{M})}	

\newcommand{\cliffib}{{\mathcal{C}\ell}(T^{*}_x \mathcal{M}, g_x)}	

\newcommand{\secextabrev}{\varGamma({\bigwedge})}

\newcommand{\secextkabrev}{\varGamma\left({\bigwedge}^k\right)}

\newcommand{\secutanb}{\varGamma(U, T^{*}\mathcal{M})}

\newcommand{\secMtanb}{\varGamma(\M, T^{*}\mathcal{M})}

\newcommand{\secuclif}{\varGamma(U, {\mathcal{C}\ell}(T^{*}\mathcal{M}))}

\newcommand{\secMclif}{\varGamma(\M, {\mathcal{C}\ell}(T^{*}\mathcal{M}))}

\newcommand{\secMext}{\varGamma(\M, {\bigwedge} (T^{*}\mathcal{M}))}

\newcommand{\secext}{\varGamma(U, {\bigwedge} (T^{*}\mathcal{M}))}

\newcommand{\secextk}{\varGamma(U, {\bigwedge}^k (T^{*}\mathcal{M}))}

\newcommand{\unsecex}{1_{\varGamma}}	

\newcommand{\secp}{\varGamma(\mathcal{P})}

\newcommand{\secs}{\varGamma(\mathcal{S})}

\begin{abstract}
\noindent The geometric product, defined by Graf on the space of differential forms,  endows the sections of the exterior bundle by a structure that is necessary to construct a Clifford algebra. The Graf product is introduced and revisited with a suitable underlying  framework that naturally encompasses a coframe in the cotangent bundle, besides the volume element centrality, the Hodge operator and the so called truncated subalgebra as well.
\end{abstract}

\maketitle
\section{Introduction}

The origins of Clifford algebras reside on the works by Clifford himself \cite{c1}, where he  introduced a quaternionic setup, employing Hamilton's quaternions, to the Grassmann's 
 theory of extensions \cite{grassmann}, deriving a framework that carries the orthogonal geometry
of an arbitrary vector space.  Thereafter, Lipschitz derived representations of rotations, implemented by complex fields, quaternionic algebras, and their higher-dimensional counterparts, together with the  Clifford (geometric) algebra and the Spin group as well. Cartan introduced representations of the Clifford algebras and the the periodicity theorem \cite{cartan}, besides the concept of pure spinors. Witt implemented Clifford algebras, studying quadratic forms constructed over arbitrary fields that have characteristic not equal to 2. The case of characteristic 2 was implemented by Chevalley \cite{Chevalley} and Riesz \cite{Riesz}, who introduced the isomorphism between Clifford and exterior algebras.

On 1962, K\"ahler introduced a new geometric product acting on exterior differential forms \cite{kahler}. This new product equipping the Grassmann exterior algebra makes it to be  isomorphic, as an associative algebra, to a Clifford algebra. This product was detailed by Graf in  Ref. \cite{Graf} and this new algebra have been named the K\"ahler-Atiyah algebra. The Clifford product that shall be presented in this text is a reformulation of the geometric product defined by Graf in 1978. Such framework was presented in Ref. \cite{cep} via the contracted wedge product. A posteriori, it has been further introduced in a manifold setup, in Ref. \cite{bab2}.

This paper constitutes a formal framework for the intrinsic algebraic structures onto which this product acts, as well as the Clifford algebra of forms. The desired product is defined, \textcolor{black}{on the} underlying sections $\varGamma({\bigwedge} (T^{*}\mathcal{M}))$ of the exterior bundle, making it into a Clifford algebra. Therefore, it is important to investigate the calculations of certain properties, with respect to   the formal Clifford product. In addition, this product provides the exterior algebra {\color{black}{with}} an interesting and effective Clifford structure, where the calculations depend only on the metric tensor and the contraction as well. 

This paper ir organized as follows: after fixing the notation and introducing fundamental algebraic and geometric features in Sect 2, in Sect. $3$ the Graf product $\diamond$ will be presented, by taking a coframe of the cotangent bundle, what makes  the algebra of differential forms to be a Clifford algebra  endowed with the Graf product. The volume element $\textbf{v}$ is then defined as the exterior product of all the elements of this coframe and therefore its centrality in $(\varGamma({\bigwedge} (T^{*}\mathcal{M})), \diamond)$ shall be proved. Besides, using the {\color{black}{Hodge operator}}, the product  $\textbf{v} \diamond \textbf{v}$ will be calculated, providing a splitting of $\varGamma({\bigwedge} (T^{*}\mathcal{M}))$. Lastly, we will study and detail the underlying and derived algebraic structures, the truncated Graf product and two prominent subalgebras of $\varGamma({\bigwedge} (T^{*}\mathcal{M}))$. \textcolor{black}{It is important to note that the main algebraic structure in this text is the set of sections of the exterior bundle, and all the Clifford structures will be considered on this set}.

\section{Preparation}

\qquad Let $(\M, g)$ be a paracompact, pseudo-Riemannian, connected manifold of signature $(p, q)$, with cotangent bundle $\pi : T^{*}\mathcal{M} \rightarrow \M$. The cotangent exterior bundle shall be denoted, as usual, by $\pi_1 : \bigwedge (T^{*}\mathcal{M}) \rightarrow \M$. All the Clifford structures will be considered in the sections of exterior bundle $\bigwedge (T^{*}\mathcal{M})$, therefore the structure of $\bigwedge (T^{*}\mathcal{M})$ shall be detailed. Firstly, the $k$-forms are defined as the sections of the $k$-power exterior bundle $\bigwedge^{k}(T^{*}\mathcal{M})$, for $k = 0, \ldots, \dim \mathcal{M} $, namely, a $k$-form \textcolor{black}{on an open set $U$ in $\M$} is an element of $\varGamma(U, \bigwedge^{k}(T^{*}\mathcal{M}))$ and a $k$-form on $\M$ is an element of $\varGamma(\M, \bigwedge^{k}(T^{*}\mathcal{M}))$. Naturally, the differential forms are defined in $\varGamma(\M, \bigwedge(T^{*} \mathcal{M}))$, being such set called in the literature \cite{extbun} as the exterior algebra of differential forms on $\mathcal{M}$. The constant function $\mathbbm{1} \in C^{\infty}(\M)$ is the unit element of $\varGamma(\M, \bigwedge(T^{*} \mathcal{M}))$, hereon denoted by $1_{\varGamma}$. It is worth to observe that a coframe to $ \bigwedge^k (T^{*} \mathcal{M})$ has $\binom{n}{k}$ elements, implying that a coframe of  $\bigwedge (T^{*} \mathcal{M})$ has $2^n$ elements, which are all the exterior products between the elements of a coframe in $\tangb$. 

For $ \dim\mathcal{M} = n$, let $\{ e_i\;|\; i \in I=\{i_1, \ldots, i_n\} \}$ be a local frame for {\color{black}{$T \mathcal{M}$}} at an open subset $U \subset \mathcal{M}$. The associated coframe for $T^{*} \mathcal{M}$ is given by the set of covector sections $\{e^i\;|\; i \in I\}$,  such that $e^i (e_j) = \delta^{i}_{j} \unsecex$. Eventually, the  convenient notations  $\varGamma(U, {\bigwedge}^k (T^{*}\mathcal{M})) = \varGamma({\bigwedge}^k)$ and $\varGamma(U, {\bigwedge} (T^{*}\mathcal{M})) = \varGamma({\bigwedge})$ shall be employed. 

The metric tensor $g:\varGamma(U, T \mathcal{M}) \times \varGamma(U, T \mathcal{M})\to\mathbb{R}$ acts on a pair $(e_i, e_j)$, whereas the metric tensor $g^{*}:\varGamma(U, T^{*} \mathcal{M}) \times \varGamma(U, T^{*} \mathcal{M})\to\mathbb{R}$ defines the reciprocal action, as
\begin{equation}
g(e_i, e_j) =: g_{ij} \in \mathbb{R} \ \text{and} \ g^{*}(e^i, e^j) =: g^{ij} \in \mathbb{R}.
\end{equation}
For an orthogonal coframe $\{e^1, \ldots, e^p, e^{p+1}, e^{p+2}, \ldots, e^n\}$, it holds $g^{ij}=0$, if $i\neq j$, {\color{black}{$g^{ii}=1$}}, if $i \in \{1, \ldots, p\}$, and {\color{black}{$g^{ii}= -1$}}, if $i \in \{p+1, \ldots, n\}$.

For a set of indexes $I_k = \{ i_1, \ldots, i_k\}$, the forms $e_{I_k} = e_{i_1 \ldots i_k} = e_{i_1} \wedge  e_{i_2} \wedge \ldots \wedge  e_{i_k}$ are defined in $\varGamma(U, \bigwedge (T \mathcal{M}))$ and $e^{I_k} = e^{i_1 \ldots i_k} = e^{i_1} \wedge  e^{i_2} \wedge \cdots \wedge  e^{i_k}$ in $\varGamma(U, \bigwedge (T^{*} \mathcal{M}))$ for $k = 1, \ldots, n$. Thus, a form $f \in \varGamma(U, \bigwedge (T^{*} \mathcal{M}))$ can be written at an open set $U$, with respect to the coframe, as:
\begin{equation}
f = \sum^{n}_{k=1} f_{I_k} e^{I_k},
\end{equation}	
where $f_{I_k}$ is a constant associated with the choices of $I_k$ in the coframe.

Naturally, the grade involution is the automorphism given by $\#(f) = \widehat{f} = \sum^{n}_{k=1} (-1)^{k} f_{I_k} e^{I_k}$ and the reversion is an anti-automorphism defined as $\mathtt{\sim}(f) = \widetilde{f} = \sum^{n}_{k=1} (-1)^{\frac{k(k-1)}{2}} f_{I_k} e^{I_k}$. 

Each fiber of exterior bundle is $\mathbb{Z}_2$-graded. \textcolor{black}{Besides, the $\mathbb{Z}_2$-grading is also well defined on $\bigwedge^k (T^{*} \mathcal{M})$,} then it is possible to consider the bundle splitting $\bigwedge (T^{*} \mathcal{M}) = \bigwedge^{+} (T^{*} \mathcal{M}) \oplus \bigwedge^{-} (T^{*} \mathcal{M})$, where $\bigwedge^{+} (T^{*} \mathcal{M})$ is the even subbundle in exterior bundle and it is constituted by even exterior algebras among the fibers. Hence, there is a induced $\mathbb{Z}_2$-graded in the sections of exterior bundle: 
\begin{eqnarray}
\varGamma^+ \left(\bigwedge\right) := \varGamma \left(U, {\bigwedge}^{+}(T^{*} \mathcal{M})\right) = \bigoplus_{k=\text{even}} \secextkabrev = \ker (\# - Id_{\secextabrev}),\\\varGamma^- \left(\bigwedge\right) := \varGamma \left(U, {\bigwedge}^{-}(T^{*} \mathcal{M}) \right) = \bigoplus_{k=\text{odd}} \secextkabrev = \ker (\# + Id_{\secextabrev}).
\end{eqnarray}

Hereon, the volume form shall be regarded as an element 
\begin{equation}\label{volu}
\textbf{v} = vol\left(\varGamma\left(U, {\bigwedge} (T^{*} \mathcal{M})\right)\right) = e^{12\ldots n} \in\varGamma\left(U, {\bigwedge}^{n} (T^{*} \mathcal{M})\right),
\end{equation}	
where $\{ e^{1}, \ldots, e^{n}\}$ is the local orthonormal coframe.

\begin{rem}
The Clifford bundle $\mathcal{C}\ell (T^{*}\mathcal{M})$ is defined as
\begin{equation}
\pi_2 :  \displaystyle{\bigsqcup_{x \in \mathcal{M}}} \mathcal{C}\ell (T_x^{*} \mathcal{M}, g_x) \rightarrow \M,
\end{equation}
where $g_x = {g\big|}_{T_x^{*} \mathcal{M}}$. \textcolor{black}{Since an orthonormal coframe has been considered, the transition functions of $\mathcal{C}\ell (T^{*}\mathcal{M})$, for open sets $U_i, U_j \subset \M$,  are
\begin{equation}
f_{ij}: U_i \cap U_j \rightarrow \mathsf{O}(n, \mathbb{R}).
\end{equation}
Note that a transition function of $T^{*}\mathcal{M}$ on $x \in U$ can be interpreted as an automorphism on the fiber $\mathcal{C}\ell (T^{*}_x\mathcal{M}, g_x)$, since $f_{ij}(x) \in \mathsf{O}(n, \mathbb{R})$. It means that the transition functions of $\mathcal{C}\ell (T^{*}\mathcal{M})$ are given by the transition functions of the cotangent bundle.}

The morphism $F$ between the Clifford bundle $\mathcal{C}\ell (T^{*}\mathcal{M})$ and the exterior bundle $\extbund$ is given by
\begin{equation}
\xymatrix{ \extbund \ar[dr]_{\pi_1} \ar[rr]^{F}  & & \mathcal{C}\ell (T^{*}\mathcal{M})  \ar[dl]^{\pi_2} \\  & \mathcal{M} &  }
\end{equation}
for projections $\pi_1$ and $\pi_2$, such that ${F\big|}_{\extfibk}$ is given by the Chevalley mapping
	\begin{equation}
\begin{array}{cccl}
& \extfibk & \rightarrow & \mathcal{C}\ell (T_x^{*} \mathcal{M}, g_x)  \\
&	D_1 \wedge \ldots \wedge D_k & \mapsto & \frac{1}{k!} \displaystyle{\sum_{\sigma \in S_k}} \operatorname{sgn}(\sigma) \ D_{\sigma(1)}  \ldots D_{\sigma(k)} 
\end{array}.
\end{equation}
Since such mapping is a linear isomorphism of vector spaces, then it is established the vector bundle isomorphism between $\extbund$ and $\mathcal{C}\ell (T^{*}\mathcal{M})$ from their fibers isomorphism.
\end{rem}

\section{An interesting non-named product}

\qquad This section is dedicated to formalize some results of the Clifford structures on the exterior bundle, for this the product defined by Graf will be chosen and it has now been rewritten via the contracted wedge product. \textcolor{black}{We also}, in this section, will study properties of this product, the volume element centrality, the Hodge operator in this setup and the truncated structure of the Clifford algebra of exterior bundle sections, as well as their subalgebras.

\begin{defn}
\upshape The \textit{contracted wedge product} of order $l$ between $f_1, f_2 \in \varGamma(U, {\bigwedge} (T^{*} \mathcal{M}))$ \cite{cep} is defined iteratively as:
\begin{eqnarray}
f_1 \wedge_0 f_2 &=& f_1 \wedge f_2\\
f_1 \wedge_l f_2 &=& \sum^{n}_{i, j = 1} g^{i j} (e_{i} \rfloor f_1) \wedge_{l-1} (e_{j} \rfloor f_2).
\end{eqnarray}
\end{defn}

The $\wedge_l$ product between a $r$-form $f_1$ and a $s$-form $f_2$ assumes the following possibilities: \\
1) $f_1 \wedge_l f_2 = 0$, if $l>r$ or $l>s$ or $l>n$; \\
2) $f_1 \wedge_l f_2$ can be nonzero, if $l \leq r$ and $l \leq s$. 

In the particular case, if $f_1$ is a $r$-form and $f_2$ is a $s$-form then 
$f_1 \wedge_1 f_2 = \sum^{n}_{i, j = 1} g^{i j} (e_{i} \rfloor f_1) \wedge (e_{j} \rfloor f_2)$ lies in $\varGamma(U, \bigwedge^{r-1+s-1}(T^{*} \mathcal{M}))$ and 
\begin{eqnarray} \nonumber
f_1 \wedge_2 f_2 &=& \sum^{n}_{a, b = 1} g^{a b} (e_{a} \rfloor f_1) \wedge_1 (e_{b} \rfloor f_2) \\ 
&=& \sum^{n}_{a, b = 1} g^{a b} \left(\sum^{n}_{i, j = 1} g^{i j} (e_{i} \rfloor e_{a} \rfloor f_1) \wedge (e_{j} \rfloor e_{b} \rfloor f_2)\right) 
\end{eqnarray}
is an element of $\varGamma(U, \bigwedge^{r-2+s-2}(T^{*} \mathcal{M}))$.

It is worth to emphasize that \textcolor{black}{the contraction $\rfloor$ maps $\varGamma(U, {\bigwedge}^{k} (T^{*} \mathcal{M}))$ on $\varGamma(U, {\bigwedge}^{k-1} (T^{*} \mathcal{M}))$}. Remembering that $e_i \rfloor e^{j} = e^j (e_i) = \delta^{j}_{i} \unsecex$, it shall be denoted \textcolor{black}{from now on} by $\delta^{j}_{i}$, for the sake of simplicity.

\begin{prop}\label{graf=ext}
If $f_1 = e^{I}$ is a $r$-form and $f_2 = e^{J}$ is a $s$-form such that $I \cap J= \emptyset$, then $f_1 \wedge_l f_2 = 0$ for all $l\geq 1$.
\end{prop}
\begin{proof}
Consider an orthonormal coframe $e^I$ to $\varGamma(U, {\bigwedge} (T^{*} \mathcal{M}))$, a $r$-form $f_1 = e^{I}$ and a $s$-form $f_2 = e^{J}$,  where the sets $I=\{i_1, \ldots, i_r\}$, $J=\{j_1, \ldots, j_s\}$ are taken such that $I \cap J= \emptyset$, then
\begin{eqnarray}\label{zero}
f_1 \wedge_1 f_2\! \!\!\!&\!\!\!\!=\!\!\!&\!\!\!\!\!\! \sum^{n}_{i, j = 1} g^{i j} (e_{i} \rfloor f_1) \wedge (e_{j} \rfloor f_2) \nonumber \\
&=& \sum^{n}_{i, j = 1} g^{i j} (e_{i} \rfloor e^{i_1 \ldots i_r}) \wedge (e_{j} \rfloor e^{j_1 \ldots j_s})  \\
&=&\!\sum^{n}_{i, j = 1}\!g^{i j}\!(\delta_{i}^{i_1}\!e^{i_2\ldots i_r}\!-\!e^{i_1}\!\wedge\!(e_i \rfloor e^{i_2\ldots i_r}))\!\wedge\!(\delta_{j}^{j_1}\!e^{j_2\ldots j_s}\!-\!e^{j_1}\!\wedge\!(e_j \rfloor e^{j_2\ldots j_s})). \nonumber
\end{eqnarray}
In Eqs. \eqref{zero}, continuing the calculations of contractions, the values of $i$ are taken according the values for $\{i_1, \ldots, i_r\}$, such that $\delta_{i}^{i_k}=1$, $k=1, \ldots, r$. Analogously, the values of $j$ are taken in $\{j_1, \ldots, j_s\}$ such that $\delta_{j}^{j_m}=1$, $m=1, \ldots, s$. As the coframe is orthonormal and the values for $i$ and $j$ are different, then $g^{ij}=0$, for all $i, j$ and,  therefore, $f_1 \wedge_1 f_2=0$. Now, suppose that the contracted wedge product of order $l>1$ between forms, on the conditions of this proposition, is zero. Let us prove that the contracted wedge product of order $l+1$ between forms of this kind is then zero. In fact, let us take $f_1=e^{i_1 \ldots i_r}$ and $f_2=e^{j_1 \ldots j_s}$ such that their superscript indexes are different. Thus,  $f_1 \wedge_{l+1} f_2 = \sum^{n}_{i, j = 1} g^{i j} (e_{i} \rfloor f_1) \wedge_l (e_{j} \rfloor f_2)$, since $e_{i} \rfloor f_1$ is a $(r-1)$-form and $e_{j} \rfloor f_2$ a $(s-1)$-form that are spanned by distinct subsets of the coframe, then the inductive hypothesis yields $(e_{i} \rfloor f_1) \wedge_l (e_{j} \rfloor f_2)=0$ and, therefore, $f_1 \wedge_{l+1} f_2 =0$.
\end{proof}

\begin{defn}
\upshape Given forms $f_1 \in  \varGamma(U, {\bigwedge}^r (T^{*} \mathcal{M}))$ and $f_2 \in  \varGamma(U, {\bigwedge}^s (T^{*} \mathcal{M}))$, $r \leq s$, the product $\diamond$ \cite{cep} between $f_1$ and $f_2$ is defined according 
\begin{equation}
f_1 \diamond f_2 = \displaystyle{\sum^{r}_{l=0}} \dfrac{(-1)^{l(r-l)+[\frac{l}{2}]}}{l!} f_1 \wedge_l f_2, 
\end{equation}
whereas 
\begin{equation}
f_2 \diamond f_1 = (-1)^{rs} \displaystyle{\sum^{r}_{l=0}} \dfrac{(-1)^{l(r-l+1)+[\frac{l}{2}]}}{l!} f_1 \wedge_l f_2.
\end{equation}
\end{defn}

\begin{thm}\label{cliff}
$ (\varGamma(U, {\bigwedge} (T^{*} \mathcal{M})), \diamond)$ is a Clifford algebra.
\end{thm}
\begin{proof}
For an arbitrary element $e^i$ in a coframe $\{ e^1, \ldots, e^n\}$, the square of elements $e^i$ in relation to the product $\diamond$ is given by 
\begin{eqnarray}\label{id} \nonumber
(e^i)^2 = e^i \diamond e^i &=&\displaystyle{\sum_{l=0}^{1}} \dfrac{(-1)^{l(1-l)+ [\frac{l}{2}]}}{l!} e^i \wedge_{l} e^i \\ \nonumber
&=& e^i \wedge e^i + e^i \wedge_1 e^i \\ \nonumber 
&=& \displaystyle{\sum_{j, k=1}^{n}} g^{jk} (e_j \rfloor e^i)  \wedge (e_k \rfloor e^i) \\ 
&=&  \displaystyle{\sum_{k=1}^{n}} g^{ik} \delta^i_k =  g^{ii} \unsecex.
\end{eqnarray}
From Eq. \eqref{id}, it is possible to conclude that the exterior algebra of the forms $\varGamma(U, {\bigwedge} (T^{*} \mathcal{M}))$, equipped with the product $\diamond$, is indeed a Clifford algebra.
\end{proof}
\noindent Note that an orthonormal coframe yields $(e^i)^2 = \pm 1_{\varGamma}$.

Due to its importance this Clifford algebra presented in the Theorem \ref{cliff} will be named hereon by \textit{Graf--Clifford algebra}. This important product between the local sections of the exterior bundle has not a proper name heretofore, up to our knowledge, and for this reason the product $\diamond$ shall be referred throughout the text as the \textit{Graf product}.

 Furthermore, the elements of an orthonormal coframe $\{e^1, \ldots, e^n\}$ anti-commute among them, with respect to the Graf product for $i \neq j$:
\begin{eqnarray}
e^i \diamond e^j &=&\displaystyle{\sum_{l=0}^{1}} \dfrac{(-1)^{l(1-l)+ [\frac{l}{2}]}}{l!} e^i \wedge_{l} e^j \nonumber\\
&=& e^i \wedge e^j + e^i \wedge_1 e^j\nonumber\\
&=& e^i \wedge e^j + \displaystyle{\sum_{r, s=1}^{n}} g^{rs} (e_r \rfloor e^i)  \wedge (e_s \rfloor e^j)\nonumber\\
&=&e^i \wedge e^j + g^{ij} \unsecex = e^i \wedge e^j \nonumber\\&=& - \  e^j \diamond e^i. 
\end{eqnarray}

\begin{rem}
The presence of the tensor components $g^{ij}$ is fundamental in the contracted wedge product for the definition of the Graf product $\diamond$. In fact, taking a manifold $(\mathcal{M}, g)$ such that $\dim \mathcal{M} = 2$, in $\secextabrev$ yields $e^{12} \diamond e^{12} = - \unsecex$ when $(p, q) = (2, 0) \ \text{or} \ (0, 2)$ and $e^{12} \diamond e^{12} = \unsecex$ when $(p, q) = (1, 1)$, where $\{e^1, e^2\}$ is the local orthonormal coframe. {\textcolor{black}{If the definition of the Graf product disregarded the metric components $g^{ij}$, then the resulting product on $(e^{12}, e^{12})$ would be equal to $0$}}.
\end{rem}

\begin{prop}\label{graf=wedge}
On the conditions of Proposition \ref{graf=ext}, the Graf product coincides with the exterior product.
\end{prop}
\begin{proof}
For a $r$-form $f_1$ and a  $s$-form $f_2$, $r\leq s$, on the conditions of Eq. \eqref{zero} yields
\begin{eqnarray}\nonumber
f_1 \diamond f_2 &=& \displaystyle{\sum^{r}_{l=0}} \dfrac{(-1)^{l(r-l)+[\frac{l}{2}]}}{l!} f_1 \wedge_l f_2 \nonumber \\
&=& f_1 \wedge f_2 + \displaystyle{\sum^{r}_{l=1}} \dfrac{(-1)^{l(r-l)+[\frac{l}{2}]}}{l!} f_1 \wedge_l f_2 \nonumber \\
&=& f_1 \wedge f_2 + 0\nonumber \\
&=& f_1 \wedge f_2.
\end{eqnarray}
\end{proof}

\textcolor{black}{Now, let us consider a $r$-form $f_1$ and a $s$-form $f_2$ such that $r \leq s$. One defines the product $\bigtriangleup$ between $f_1$ and $f_2$ as
\begin{equation}
f_1 \bigtriangleup  f_2 = \displaystyle{\sum^{r}_{l=0}} \dfrac{(-1)^{l(r-l+1)+[\frac{l}{2}]}}{l!} f_1 \diamond_l f_2, 
\end{equation}
where $\diamond_l$ is the \textit{contracted Graf product} of order $l$ between $f_1$ and $ f_2$, which is  iteratively defined by
\begin{eqnarray}
f_1 \diamond_0 f_2 &=& f_1 \diamond f_2\\
f_1 \diamond_l f_2 &=& \sum^{n}_{i, j = 1} g^{i j} (e_{i} \rfloor f_1) \diamond_{l-1} (e_{j} \rfloor f_2).
\end{eqnarray}}

\textcolor{black}{\begin{prop}
The product $\bigtriangleup$ restricted to $\secutanb$ is the wedge product.
\end{prop}
\begin{proof}
Let $e^k, e^m$ be elements in the coframe of $T^* \M$ for $k\neq m$, thus
\begin{eqnarray} \nonumber
e^k \bigtriangleup e^m &=& e^k \diamond e^m + (-1)^{1(2-1)} e^k \diamond_1 e^m \\ \nonumber
&=& e^k \wedge e^m - \sum^{n}_{i, j = 1} g^{i j} (e_{i} \rfloor e^k) \diamond (e_{j} \rfloor e^m) \\
&=& e^k \wedge e^m - \underbrace{g^{k m}}_{0} \delta_{k}^{k} \unsecex \wedge \delta_{m}^{m} \unsecex  = e^k \wedge e^m.
\end{eqnarray}
Furthermore, if $k=m=$ then
\begin{eqnarray} \nonumber
e^k \bigtriangleup e^k = e^k \diamond e^k - e^k \diamond_1 e^k &=& g^{k k} \unsecex - g^{k k} \delta_{k}^{k} \unsecex \wedge \delta_{k}^{k} \unsecex \\ 
&=& g^{k k} \unsecex - g^{k k} \unsecex = 0.
\end{eqnarray}
Hence, the product $\bigtriangleup$ over sections of $\tangb$, which is defined from the Graf product is actually the wedge product $\wedge$.
\end{proof}
}

\textcolor{black}{Let $e^i$ be a section of $\tangb$, since $\secMtanb$ can be included in $\secMext$ and considering the identification
\begin{equation}
\extbund \stackrel{F}{\longrightarrow} \clifbund
\end{equation}
then the mapping $F \circ e^i$ lies in $\secMclif$. This means that a coframe $\{e^1, \ldots, e^n\}$ of $\tangb$ can be included in $\secMclif$, this is, there is an identification
\begin{equation}
\secMext \longleftrightarrow \secMclif,
\end{equation}
once $\{e^1, \ldots, e^n\}$ generates $\secMext$.
Besides, if $e^i \in \secMtanb$ then the composition $F \circ e^i$ can be identified  with $e^i$, namely, $e^i$ can be included in $\secMclif$.
}

Let us consider $x \in U \subseteq \M$, thus $\{e^1(x), \ldots, e^n(x)\}$ generates the fiber $\cliffib$, a Clifford algebra, where $g_x = {g\big|}_{\tanfib}$, yielding 
\begin{equation}
(e^i (x))^2 =  g_x (e^i (x), e^i (x)).
\end{equation}
Thus it is possible to induce a product $\ast$ in $\secuclif$, which is defined as follows:
\begin{equation}
e^i \ast e^i = g (e^i, e^i),
\end{equation}
whereas
\begin{equation}
e^i (x) \bullet e^i (x) =  {g\big|}_{\tanfib} (e^i (x), e^i (x)), \ \forall x \in U,
\end{equation}
where $\bullet$ is the Clifford product in the fiber $\cliffib$. \textcolor{black}{These two products $\ast$ and $\bullet$ are essentially different, since they regard distinct structures. On the one hand, the product $\bullet$ is the Clifford product on the algebra $\cliffib$, whereas the product $\ast$ is only the generalized product defined on $\clifbund$-sections, whose  structure is not the same as of $\cliffib$. Such product $\ast$ was merely defined  to clarify that $\secuclif$ has an underlying geometric structure of a Clifford algebra,  provided by the algebraic Clifford structure of the fibers.}

\textcolor{black}{Since $\secext$ is identified with $\secuclif$, then a local coframe of $\clifbund$ can be obtained from the local coframe $\{e^{i_1}\wedge \cdots \wedge e^{i_k} \ | \ i_1< \cdots < i_k; \ k=1, \ldots, n\}$, i. e., the set $\{e^{i_1}\ast \cdots \ast e^{i_k} \ | \ i_1< \cdots < i_k; \ k=1, \ldots, n\}$ is a coframe for $\clifbund$. Therefore, $(\secuclif, \ast)$ has a structure of Clifford algebra.
}

\begin{rem}
\textcolor{black}{It is clear that both $(\secext, \diamond)$ and $(\secuclif, \ast)$ are Clifford algebras. Both these Clifford algebras do exist by virtue of the geometric structure considered. Regarding the inherent algebraic structure, the exterior algebra is in fact, not a Clifford algebra. However, it can be turned into a Clifford algebra as long as one considers the Graf product and the geometric structure of the sections of $\extbund$ and $\clifbund$.}
\end{rem}

From now on we will prove some results about the volume element. Since the coframe $\{e^1, \ldots, e^n\}$ is orthonormal, then the matrix that represents $g^*$ reads $\operatorname{diag}(g^{11}, g^{22}, \ldots, g^{nn})$.

\begin{prop}
The volume element $\textbf{v}$ squares according to 
\begin{equation}\label{volvol}
\!\!\textbf{v} \diamond \textbf{v} = \left\{ \begin{array}{ll}
+ \unsecex, &\!\!\!\! \text{if} \ p-q \equiv_4 0, 1 \Leftrightarrow p-q \equiv_8 0, 1, 4, 5 \\
- \unsecex, & \!\!\!\! \text{if} \ p-q \equiv_4 2, 3 \Leftrightarrow p-q \equiv_8 2, 3, 6, 7 
\end{array} \right..
\end{equation}
\end{prop}
\begin{proof}
Let us use the definition of the Graf product on $(\textbf{v}, \textbf{v})$:
\begin{eqnarray}
\textbf{v} \diamond \textbf{v}  &=& (e^1 \wedge \cdots \wedge e^n)\diamond (e^1 \wedge \cdots \wedge e^n) \nonumber\\
&=& \displaystyle{\sum_{l=0}^{n}} \dfrac{(-1)^{l(n-l)+ [\frac{l}{2}]}}{l!} \ (e^1 \wedge \cdots \wedge e^n) \wedge_{l} (e^1 \wedge \cdots \wedge e^n)\nonumber \\
&=& e^1\!\wedge\!\cdots\!\wedge\!e^n\!\wedge\!e^1\!\wedge\!\cdots\!\wedge\!e^n \!+\!(-1)^{n-1} e^{12\ldots n} \wedge_1 e^{12\ldots n}\nonumber\\
&&\!+\!\frac{1}{2} (-1)^{2(n-2)+1} e^{12\ldots n} \wedge_2 e^{12\ldots n}\!+\! \frac{1}{3!} (-1)^{3(n-3)+1} e^{12\ldots n} \wedge_3 e^{12\ldots n}\nonumber\\
&&\!+\!\cdots\!+\!\frac{1}{n!} (-1)^{[\frac{n}{2}]} e^{12\ldots n} \wedge_n e^{12\ldots n}.
\end{eqnarray}
For $l=0, 1, 2, \ldots, n-1$, the terms $g^{ii}$  are  always accompanied by a term $g^{jk}$, with $k\neq j$, thus making $\textbf{v} \wedge_l \textbf{v}$ to vanish. Nevertheless, the contracted wedge product of order $n$ on $(\textbf{v}, \textbf{v})$ is $\textbf{v} \wedge_n \textbf{v} = n! \ g^{11}g^{22} \ldots g^{nn} \ \unsecex$. Therefore
\begin{equation}
\textbf{v} \diamond \textbf{v} = (-1)^{[\frac{n}{2}]} \ g^{11}g^{22} \ldots g^{nn} \ \unsecex = (-1)^{[\frac{n}{2}]} (-1)^q \ \unsecex.
\end{equation}

Since $q+\left[\frac{n}{2}\right] \equiv_2 \left\{ \begin{array}{ll}
\frac{p-q}{2}, & \ \text{if} \ n \ \text{is even} \\
\frac{p-q-1}{2}, & \ \text{if} \ n \ \text{is odd}
\end{array} \right. $, for $p-q \equiv_4 0, 1$, it is possible obtain the even values of $\frac{p-q}{2}$ and $\frac{(p-q)-1}{2}$. For $p-q \equiv_4 2, 3$,  these fractions present odd values. Hence, it follows that
\begin{equation}\label{volvol} \nonumber
\!\!\textbf{v} \diamond \textbf{v} = (-1)^{[\frac{n}{2}]+q}\ \unsecex = \left\{ \begin{array}{ll}
+ \unsecex, &\!\!\!\! \text{if} \ p-q \equiv_4 0, 1 \Leftrightarrow p-q \equiv_8 0, 1, 4, 5 \\
- \unsecex, & \!\!\!\! \text{if} \ p-q \equiv_4 2, 3 \Leftrightarrow p-q \equiv_8 2, 3, 6, 7 
\end{array} \right..
\end{equation}
\end{proof}

The next proposition concerns about the necessary conditions for $\textbf{v}$ to be central:

\begin{prop}
If $n= \dim \M$ is odd, then $\textbf{v}$ is central on the Graf--Clifford algebra.
\end{prop}
\begin{proof}
Given $f \in \varGamma(U, {\bigwedge}^m (T^{*}\mathcal{M}))$, the result of $f \diamond \textbf{v}$ has the term $f \wedge_i \textbf{v}$ equals zero, for $i= 0, \ldots, m-1$, when $e^1 \wedge \ldots \wedge e^n$ is contracted $k$ times, appearing all possible $(n-k)$-forms in the final expression, $1<k<m$, that vanish when multiplied with $f$ contracted $k$ times.
Therefore, in the term $f \wedge_{k} \textbf{v}$ it appears forms $e^{I}$, where $I$ has at least two repeated elements and, when it does not happen, the coefficients that follow the forms are of the type $g^{ij}=0$ for $i\neq j$. If $k>m$, then $f \wedge_k \textbf{v} = 0$, being also $f$, contracted $m$ times, a multiple of $\unsecex$. Hence,
\begin{eqnarray}\label{fvol}
f \diamond \textbf{v} &=& \displaystyle{\sum_{l=0}^{m-1}} \dfrac{(-1)^{l(m-l)+ [\frac{l}{2}]}}{l!} \ f \wedge_{l} \textbf{v} + \dfrac{(-1)^{m(m-m)+ [\frac{m}{2}]}}{m!} \ f \wedge_{m} \textbf{v}\nonumber \\
&=& \frac{1}{m!} (-1)^{[\frac{m}{2}]} f \wedge_m \textbf{v}.
\end{eqnarray}
On the other hand,
\begin{eqnarray}
\textbf{v} \diamond f &=& (-1)^{nm} \dfrac{(-1)^{m(m-m+1)+ [\frac{m}{2}]}}{m!} f \wedge_m \textbf{v} \nonumber \\
&=& \frac{1}{m!} (-1)^{mn+m+[\frac{m}{2}]} f \wedge_m \textbf{v}.
\end{eqnarray}
When $n$ is odd the values of $mn+m$ are even for all $m$, thus 
\begin{eqnarray}\label{comvol}
\textbf{v} \diamond f &=& \frac{1}{m!} (-1)^{mn+m} (-1)^{[\frac{m}{2}]} f \wedge_m \textbf{v} \nonumber \\
&=& \frac{1}{m!} (-1)^{[\frac{m}{2}]} f \wedge_m \textbf{v} \nonumber\\&=& f \diamond \textbf{v}.
\end{eqnarray} 
Hence, $\textbf{v}$ commutes with $f \in \varGamma(U, {\bigwedge}(T^{*}\mathcal{M}))$ if the dimension of $\mathcal{M}$ is odd.
\end{proof}

\begin{defn}
\upshape The \textit{Hodge operator} $\star$ can be defined in relation to the Graf product as
\begin{equation}
\begin{array}{cccc} \star : & \varGamma(U, {\bigwedge}^r(T^{*}\mathcal{M}))  & \rightarrow & \varGamma(U, {\bigwedge}^{n-r}(T^{*}\mathcal{M})) \\
& f & \mapsto &  f \diamond \textbf{v}
\end{array}.
\end{equation}
\end{defn}

This Hodge operator is well-defined, since from Eq. \eqref{fvol} holds
\begin{equation}
\star f  = \frac{1}{r!} (-1)^{[\frac{r}{2}]} f \wedge_r \textbf{v},
\end{equation}
which in turn is a $(n-r)$-form because $f \wedge_r \textbf{v}$ lies in $\varGamma(U, {\bigwedge}^{n-r}(T^{*}\mathcal{M}))$. This chosen definition is an adaptation that disregards the reversion of $f$ in the definition used in \cite{bab1}.

Observe that the next identities hold:
\begin{flalign}
&\star \unsecex = \unsecex \diamond \textbf{v} = \frac{(-1)^{0(0-0)+[\frac{0}{2}]}}{0!} \unsecex \wedge_{0} \textbf{v} = \unsecex \wedge \textbf{v} = \textbf{v},&
\end{flalign}
\begin{flalign}
&\star \textbf{v} = \textbf{v} \diamond \textbf{v} = \left\{ \begin{array}{ll}
+ \unsecex, & \ \text{if} \ p-q \equiv_4 0, 1 \Leftrightarrow p-q \equiv_8 0, 1, 4, 5 \\
- \unsecex, & \ \text{if} \ p-q \equiv_4 2, 3 \Leftrightarrow p-q \equiv_8 2, 3, 6, 7 
\end{array} \right. .&
\end{flalign}

Besides, is $f$ lies in the Graf--Clifford algebra, then the Hodge operator squared reads:
\begin{equation}
\begin{array}{lll}
\star^2(f) &=& \star (f \diamond \textbf{v})\\ 
&=& f \diamond (\textbf{v} \diamond \textbf{v}) \\
&=& (\textbf{v} \diamond \textbf{v}) \wedge f\\
&=& \left\{ \begin{array}{ll}
+ \unsecex \wedge f = f, &  \text{if} \ p-q \equiv_8 0, 1, 4, 5 \\
- \unsecex \wedge f = -f, &  \text{if} \  p-q \equiv_8 2, 3, 6, 7 
\end{array} \right. .
\end{array}
\end{equation}
Thereat, defining $\varGamma_{\star^2}^{\pm} = \{f \in \varGamma(U, {\bigwedge} (T^{*}\mathcal{M})) \ | \ \star^2 f = \pm f \}$, the Hodge operator provides another $\mathbb{Z}_2$-grading for the Graf--Clifford algebra.

Now, consider the elements $p_{\pm} := \frac{1}{2} (\unsecex \pm \textbf{v})\in\varGamma({\bigwedge}^0) \oplus \varGamma({\bigwedge}^n)$, they satisfy the properties
\begin{equation} \label{x}
p_+ + p_- = \unsecex,
\end{equation}
\begin{equation}\label{y}
\begin{array}{lll}
p_\pm \diamond p_\pm &=& \left\{ \begin{array}{ll}
\frac{1}{2} (\unsecex \pm \textbf{v}), &  \text{if} \ p-q \equiv_8 0, 1, 4, 5 \\
\pm \frac{1}{2}  \textbf{v}, &  \text{if} \  p-q \equiv_8 2, 3, 6, 7 
\end{array} \right. ,
\end{array}
\end{equation}

\begin{equation}\label{z}
\begin{array}{lll}
p_\pm \diamond p_\mp &=& \left\{ \begin{array}{ll}
0, &  \text{if} \ p-q \equiv_8 0, 1, 4, 5 \\
\frac{1}{2}  \unsecex, &  \text{if} \  p-q \equiv_8 2, 3, 6, 7 
\end{array} \right. .
\end{array}
\end{equation}
Then it follows that the right regular representation of an element $f \in \varGamma(U, {\bigwedge} (T^{*}\mathcal{M}))$ by $p_\pm$ is defined by:
\begin{equation}
R_{p_\pm} (f) = P_\pm (f) := f \diamond p_\pm = \frac{1}{2} (f \pm f \diamond \textbf{v}) = \frac{1}{2} (f \pm \star f),
\end{equation}
namely,
\begin{equation}
P_\pm = \frac{1}{2} (Id_{\varGamma\left(\bigwedge\right)} \pm \star).
\end{equation}
By Eqs. \eqref{x}, \eqref{y} and \eqref{z}, it follows that $P_+ + P_- = Id_{\varGamma\left(\bigwedge\right)}$ and $P_\pm^2 = P_\pm$, $P_\pm \circ P_\mp = 0$. 

The sets 
\begin{equation}
\varGamma_\pm := P_\pm \left(\varGamma\left(\bigwedge\right)\right) = \varGamma\left(\bigwedge\right) \diamond p_\pm
\end{equation}
are not always  subalgebras of $(\secextabrev, \diamond)$.

\begin{prop}
If $p-q \equiv_8 0, 1, 4, 5$, then $\varGamma(U, {\bigwedge} (T^{*}\mathcal{M})) = \varGamma_+ \oplus \varGamma_-$.
\end{prop}
\begin{proof}
If $p-q \equiv_8 0, 1, 4, 5$, for $f_1 \in \varGamma_+$ and $f_2 \in \varGamma_-$ there exists $g, h \in \varGamma\left(\bigwedge\right)$, such that $f_1 = P_+ (h) = \frac{1}{2} (h + \star h)$, and $f_2 = P_- (g) =  \frac{1}{2} (g - \star g)$. Hence, 
\begin{eqnarray}
\star f_1 &=& \frac{1}{2} (\star h + \star^2 h)= \frac{1}{2} (\star h +  h) = f_1\\\star f_2 &=& \frac{1}{2} (\star g - \star^2 g)= \frac{1}{2} (\star g -  g) = -f_2,
\end{eqnarray}
meaning that for $p-q \equiv_8 0, 1, 4, 5$, it implies that
\begin{equation}
\varGamma_\pm = \left\{f \in \varGamma\left(\bigwedge \right ) |  \star f= \pm f \right\}
\end{equation} 
and for this reason there is a splitting $\varGamma(U, {\bigwedge} (T^{*}\mathcal{M})) = \varGamma_+ \oplus \varGamma_-$.
\end{proof}

\begin{rem}
\textcolor{black}{If $p-q \equiv_8 2, 3, 6, 7$, this splitting does not necessarily exist over $\mathbb{R}$.}
\end{rem}  

In addition, this implies that if $f \in \varGamma_\pm$ and $p-q \equiv_8 0, 1, 4, 5$, then 
\begin{equation}
P_\pm (f) = \frac{1}{2}(f \pm \star f) = \frac{1}{2} (f\pm (\pm f))=f.
\end{equation}

\begin{prop}\label{p311}
If $n$ is odd and $p-q \equiv_8 0, 1, 4, 5$, then 

\noindent i) $P_\pm$ are endomosphisms,

\noindent ii) $\varGamma_\pm$ are subalgebras of $\secextabrev$ and

\noindent iii) the units of $\varGamma_\pm$ are $p_\pm$.
\end{prop}
\begin{proof}
i) and ii) 
For $f_1, f_2 \in \varGamma\left(\bigwedge\right)$, note that
\begin{eqnarray}\label{end}
\!\!\!\!\!\!\!\!\!\!\!\!\!\!P_\pm (f_1) \diamond P_\pm (f_2)\!&\!=\!&\!\frac{1}{4} (f_1 \pm f_1 \diamond \textbf{v})\diamond(f_2 \pm f_2 \diamond \textbf{v}) \nonumber\\
\!&\!=\!&\!\frac{1}{4} f_1\!\diamond\!f_2 \pm  \frac{1}{4} f_1\!\diamond\!f_2\!\diamond\!\textbf{v} \pm \frac{1}{4} f_1\!\diamond\!\textbf{v} \!\diamond\!f_2 + \frac{1}{4} f_1\!\diamond\!\textbf{v}\!\diamond\!f_2\!\diamond\!\textbf{v},
\end{eqnarray}
which is not necessarily the image of an element in $\secextabrev$ by $P_{\pm}$.

Regarding Eq. \eqref{end}, if $\textbf{v}$ is central ($n$ odd) and $p-q \equiv_8 0, 1, 4, 5$, thus $P_\pm (f_1) \diamond P_\pm (f_2) = P_\pm (f_1 \diamond f_2)$, namely, $P_\pm$ are endomorphisms of $\varGamma(U, {\bigwedge} (T^{*}\mathcal{M}))$ and under such conditions the sets $(\varGamma_\pm, \diamond)$ are closed in relation to the Graf product and therefore are subalgebras of $\varGamma(U, {\bigwedge} (T^{*}\mathcal{M}))$.

\noindent iii) A good attempt the units of $\varGamma_\pm$ is given by the following elements
\begin{eqnarray}
P_\pm (\unsecex) &=& \frac{1}{2} (Id_{\varGamma\left(\bigwedge\right)}(\unsecex) \pm \star(\unsecex)) \nonumber \\
&=& \frac{1}{2} (\unsecex \pm \unsecex \diamond \textbf{v}) = \frac{1}{2} (\unsecex \pm \textbf{v}) = p_\pm.
\end{eqnarray}
Indeed they are: for $f \in \secextabrev$, since $p-q \equiv_8 0, 1, 4, 5$ yields
\begin{eqnarray}
P_\pm (f) \diamond  p_\pm &=& \frac{1}{4} (f \pm f \diamond \textbf{v} \pm f \diamond \textbf{v} + f \diamond \textbf{v} \diamond \textbf{v}) \nonumber \\
&=& \frac{1}{2} (f \pm f \diamond \textbf{v}) = P_\pm (f).
\end{eqnarray}
and here it is important that $n$ be odd
\begin{eqnarray}
p_\pm \diamond P_\pm (f) &=& \frac{1}{4} (f \pm \textbf{v} \diamond f \pm f \diamond \textbf{v} +  \textbf{v} \diamond f  \diamond \textbf{v}) \nonumber \\
&=& \frac{1}{2} (f \pm f \diamond \textbf{v}) = P_\pm (f).
\end{eqnarray}
\end{proof}


Consider the splitting
\begin{equation}
\varGamma\left(U, {\bigwedge} (T^{*}\mathcal{M})\right) = \varGamma_L \oplus \varGamma_U := \bigoplus_{k=0}^{[\frac{n}{2}]} \varGamma\left({\bigwedge}^k\right) \oplus \bigoplus_{k=[\frac{n}{2}]+1}^{n} \varGamma\left({\bigwedge}^k\right).
\end{equation}
The endomorphisms $P_\pm$ are emulated under this splitting: the {lower truncation} $P_L (f)= f_L := \sum_{k=0}^{[\frac{n}{2}]} f_{I_k} e^{I_k}$ and the {upper truncation} $P_U (f)= f_U :=\sum_{k=[\frac{n}{2}]+1}^{n} f_{I_k} e^{I_k}$. In this notation, the split spaces are $\varGamma_L = P_L (\varGamma\left(\bigwedge\right))$ and $\varGamma_U = P_U (\varGamma\left(\bigwedge\right))$ and these new endomorphisms are central, commuting, idempotents. 

For $f \in \varGamma(U, {\bigwedge} (T^{*}\mathcal{M}))$, it is established that $\star f = \pm f$ if $p-q \equiv_8 0, 1, 4, 5$, implying that $\pm(f_L + f_U )= \pm f = \star f = \star (f_L + f_U) = \star f_L + \star f_U$. Since $f_L$ and $\star f_U$ are elements of $ \bigoplus_{k=0}^{[\frac{n}{2}]} \varGamma({\bigwedge}^k)$ and $f_U, \star f_L \in \bigoplus_{k=[\frac{n}{2}]+1}^{n} \varGamma({\bigwedge}^k)$, it  follows that
\begin{equation}
f_L = \pm \star{f_U} \ \ \text{and} \ \ f_U = \pm \star f_L.
\end{equation}
Hence,
\begin{equation}\label{f2}
f = f_L + f_U = f_L \pm \star f_L = 2 P_\pm (f_L) = 2 P_\pm (P_L(f)) =  P_\pm (2 P_L(f)).
\end{equation}
Since $\varGamma_L$ is not a subalgebra of the Graf--Clifford algebra, it is necessary to define a new product:

\begin{defn}
\upshape The \textit{truncated Graf product} is defined as follows
\begin{equation}\label{tgp}
\begin{array}{cccl} \blkdiam_\pm : & \secextabrev \times \secextabrev  & \rightarrow & \varGamma_L \subset \secextabrev\\
& (f_1, f_2) & \mapsto &  2 P_L (P_\pm (f_1) \diamond P_\pm (f_2))
\end{array}.
\end{equation}
\end{defn}

Thus the preservation problem is solved. Indeed, 
\begin{equation}\label{P+} 
P_\pm (f_1 \blkdiam_\pm f_2) = P_\pm (2 P_L (P_\pm (f_1) \diamond P_\pm (f_2))) \stackrel{\eqref{f2}}{=} P_\pm (f_1) \diamond P_\pm (f_2).
\end{equation}
\textcolor{black}{On the conditions of the Proposition \ref{p311}, the mappings $P_\pm$ are endomorphisms. Hence, Eqs. \eqref{tgp} and \eqref{P+} can be refined as follows, }
\begin{equation}\label{teste}
f_1 \blkdiam_\pm f_2 = 2 P_L (P_\pm (f_1 \diamond f_2)), \ \forall f_1, f_2 \in \secextabrev,
\end{equation} and  
\begin{equation} 
P_\pm (f_1 \blkdiam_\pm f_2) = P_\pm (f_1 \diamond f_2).
\end{equation}

\begin{rem}
Whilst the set $\varGamma_L$ is not necessarily a subalgebra of the Graf--Clifford algebra, it is a subalgebra of the \textit{Truncated algebra} $(\secextabrev, \blkdiam_\pm)$.
\end{rem}

\begin{prop}
If $n$ is odd and $p-q \equiv_8 0, 1, 4, 5$, then the unit section $1_\varGamma$ is the unit of $(\varGamma_L, \blkdiam_\pm)$.
\end{prop}
\begin{proof}
For an arbitrary $f \in \varGamma_L$ yields
\begin{eqnarray}\label{unitLr}
f \blkdiam_\pm \unsecex &=& 2 P_L (P_\pm (f) \diamond P_\pm (\unsecex)) \nonumber \\
&=& 2 P_L (P_\pm (f) \diamond p_\pm) \nonumber \\
&=& 2 P_L (\frac{1}{4} (f \pm f \diamond \textbf{v}) \diamond (\unsecex \pm \textbf{v})) \nonumber \\
&=& \frac{1}{2} P_L (f \pm f \diamond \textbf{v} \pm f \diamond \textbf{v} + f \diamond \textbf{v} \diamond \textbf{v}) \nonumber \\
&=& \frac{1}{2} (f + f \diamond (\textbf{v} \diamond \textbf{v})) \pm P_L( f \diamond \textbf{v}) \nonumber \\
&=& \left\{ \begin{array}{ll}
f \pm P_L( f \diamond \textbf{v}),  \ \text{if} \ p-q \equiv_8 0, 1, 4, 5 \\
\pm P_L( f \diamond \textbf{v}), \   \text{if} \  p-q \equiv_8 2, 3, 6, 7 
\end{array} \right. .
\end{eqnarray}
If $n$ is odd, then $ f \diamond \textbf{v}$ is at least a $(n - [\frac{n}{2}])$-form such that $n - [\frac{n}{2}] \neq [\frac{n}{2}] $. It implies, in this case, $P_L( f \diamond \textbf{v})=0$. Hence, when $p-q \equiv_8 0, 1, 4, 5$ and $n$ is odd, it yields
\begin{equation} \label{unitL}
f \blkdiam_\pm \unsecex = f, \ \forall f \in \varGamma_L.
\end{equation}
On the other hand, if $n$ is odd then $\textbf{v}$ is central, hence 
\begin{eqnarray}\label{unitLl}
\unsecex \blkdiam_\pm f &=& 2 P_L (p_\pm \diamond P_\pm (f)) \nonumber \\
&=& 2 P_L \left(\frac{1}{4} (\unsecex \pm \textbf{v})  \diamond (f \pm f \diamond \textbf{v})\right) \nonumber \\
&=& \frac{1}{2} P_L (f \pm \textbf{v} \diamond f \pm f \diamond \textbf{v} +  \textbf{v} \diamond f  \diamond \textbf{v}) \nonumber \\
&=& \frac{1}{2} (f + \textbf{v} \diamond f  \diamond \textbf{v}) \pm \frac{1}{2} P_L( f \diamond \textbf{v} + \textbf{v} \diamond f) \nonumber \\
&=& \frac{1}{2} (f +   f  \diamond (\textbf{v} \diamond \textbf{v})) \pm  P_L( f \diamond \textbf{v}) \nonumber \\
&=& \left\{ \begin{array}{ll}
f \pm P_L( f \diamond \textbf{v}),  \ \text{if} \ p-q \equiv_8 0, 1, 4, 5 \\
\pm P_L( f \diamond \textbf{v}), \   \text{if} \  p-q \equiv_8 2, 3, 6, 7 
\end{array} \right. .
\end{eqnarray}
Then, under the conditions that regard Eq. \eqref{unitL}, it yields $\unsecex \blkdiam_\pm f=f$.
\end{proof}

\begin{prop}
If $n$ is odd and $p-q \equiv_8 0, 1, 4, 5$, then $(\varGamma_L, \blkdiam_\pm) \cong (\varGamma_\pm, \diamond)$.
\end{prop}
\begin{proof}
Consider the mappings ${P_\pm \big|}_{\varGamma_L}: (\varGamma_L, \blkdiam_\pm) \rightarrow (\varGamma_\pm, \diamond)$ and ${2 P_L \big|}_{\varGamma_\pm}: (\varGamma_\pm, \diamond) \rightarrow (\varGamma_L, \blkdiam_\pm)$, which are homomorphisms of subalgebras. In fact, for arbitrary $f_1, f_2 \in \varGamma_L$
\begin{eqnarray}
{P_\pm \big|}_{\varGamma_L}(f_1 \blkdiam_\pm f_2) &=& {P_\pm \big|}_{\varGamma_L} (2 P_L (P_\pm (f_1) \diamond P_\pm (f_2))) \nonumber \\
&=& {P_\pm \big|}_{\varGamma_L} (2 P_L (P_\pm (f_1 \diamond f_2))) \nonumber \\
&=& P_\pm (f_1 \diamond f_2) \nonumber \\
&=& P_\pm (f_1) \diamond  P_\pm (f_2) \nonumber \\
&=& {P_\pm \big|}_{\varGamma_L} (f_1) \diamond {P_\pm \big|}_{\varGamma_L}(f_2).
\end{eqnarray}
and for arbitrary $f_1, f_2 \in \varGamma_\pm$
\begin{eqnarray}
{2 P_L \big|}_{\varGamma_\pm}(f_1 \diamond f_2) &=& {2 P_L \big|}_{\varGamma_\pm}(f_1 \diamond f_2) \nonumber \\
&=& 4 \left(2 P_L \left(\frac{1}{2} f_1 \diamond \frac{1}{2} f_2\right)\right) \nonumber \\
&=& 4 (2 P_L (P_\pm(P_L(f_1) \diamond P_\pm(P_L(f_2)))) \nonumber \\
&=& 4 (2 P_L (P_\pm({P_L \big|}_{\varGamma_\pm}(f_1) \diamond P_\pm({P_L \big|}_{\varGamma_\pm}(f_2)))) \nonumber \\
&=& 4 {P_L \big|}_{\varGamma_\pm}(f_1) \blkdiam_\pm {P_L \big|}_{\varGamma_\pm}(f_2) \nonumber \\
&=& {2 P_L \big|}_{\varGamma_\pm}(f_1) \blkdiam_\pm {2 P_L \big|}_{\varGamma_\pm}(f_2).
\end{eqnarray}

The homomorphism ${P_\pm \big|}_{\varGamma_L}$ is injective, given $f_1, f_2 \in \varGamma_L$, if ${P_\pm \big|}_{\varGamma_L}(f_1)= {P_\pm \big|}_{\varGamma_L}(f_2)$, then $P_\pm (P_L(f_1)) = P_\pm (P_L(f_2)) \Rightarrow \frac{1}{2} f_1 = \frac{1}{2} f_2 \Rightarrow f_1=f_2$. In addition, this mapping is surjective: given $f \in \varGamma_\pm$, the element $2{P_L \big|}_{\varGamma_\pm} (f) \in \varGamma_L$ is such that $f= P_\pm (2 P_L(f)) = {P_\pm \big|}_{\varGamma_L}(2{P_L \big|}_{\varGamma_\pm} (f))$, which implies that $\operatorname{Im} {P_\pm \big|}_{\varGamma_L} = \varGamma_\pm$. Thus, ${P_\pm \big|}_{\varGamma_L}$ is an isomorphism of subalgebras. \end{proof}

Therefore if $n$ is odd and $p-q \equiv_8 0, 1, 4, 5$ the truncated subalgebra $(\varGamma_L, \blkdiam_\pm)$ is isomorphic to $(\varGamma_\pm, \diamond)$, such structures will be very useful to find new pinor and spinor classes.

\section{Conclusions}

The main goal here has been to introduce the Graf product that endows a Clifford algebra, emulating the Riezs' construction of Clifford algebras, the K\"ahler-Atiyah algebra, into a Clifford and exterior bundle contexts. 
The volume element centrality is discussed in  $(\varGamma(U, {\bigwedge} (T^{*}\mathcal{M})), \diamond)$. The Hodge operator has fundamental importance in the developed setup, since it splits the space $\varGamma(U, {\bigwedge} (T^{*}\mathcal{M}))$ and 
the truncated structures can be, \textcolor{black}{then}, defined. Hence, the truncated Graf product has been introduced. Besides a thorough formalism that excels the K\"ahler-Atiyah algebra, this framework is a very useful one to pave the 
 real pinor bundles on $(\M,g)$, defined as a vector bundle equipped with a morphism of bundles of algebras   that play the role of a bundle of modules over the K\"ahler-Atiyah bundle, namely, the exterior bundle endowed with the geometric product emulated by the Graf one. This is a fundamental framework that pave the way to define pinor and spinor bundles, whose paramount structural importance and huge spectrum of applications has been studied in Ref. \cite{bab1,bab2}, 
where new classes of spinor fields have been derived in several dimensions and signatures in Refs. \cite{Bonora:2014dfa,Bonora:2015ppa,deBrito:2016qzl,Ablamowicz:2014rpa}. We are going to employ the setup here developed to studied further classes of spinor fields in manifolds of signature (9,0), being beyond of the scope of this work.

\subsection*{Acknowledgements}{RL thanks to CAPES and RdR is grateful to CNPq (grant No. 303293/2015-2) and to FAPESP (grant No. 2017/18897-8), for partial financial support.}

\end{document}